\documentclass[11pt, reqno]{amsart}
\usepackage{amsmath,amsthm,amssymb,mathrsfs}
\usepackage[abbrev,non-sorted-cites]{amsrefs}
\usepackage{color,graphicx}
\usepackage{verbatim}
\usepackage{datetime}
\usepackage{hyperref}

\theoremstyle{plain}
  \newtheorem{thm}{Theorem}[section]
  \newtheorem{lem}[thm]{Lemma}
  
  \newtheorem{cor}[thm]{Corollary} 
	\newtheorem*{thm*a}{Theorem A}
	\newtheorem*{thm*b}{Theorem B}
\theoremstyle{definition}

  \newtheorem{rmk}[thm]{Remark}
  \newtheorem*{ack*}{Acknowledgement}
  \newtheorem*{ques*}{Question}
\theoremstyle{plain}

\numberwithin{equation}{section}

\newcommand\ip[2]{\langle{#1},{#2}\rangle}

\newcommand\oh{\frac{1}{2}}
\newcommand\dd{{\mathrm d}}

\newcommand\w{\wedge}
\newcommand\sm{\sigma}
\newcommand\dt{\delta}

\newcommand\vep{\varepsilon}
\newcommand\vph{\varphi}

\newcommand\kp{\kappa}
\newcommand\af{\alpha}
\newcommand\bt{\beta}

\newcommand\ld{\lambda}

\newcommand\Sm{\Sigma}
\newcommand\Gm{\Gamma}
\newcommand\Ld{\Lambda}

\newcommand\Dt{\Delta}

\newcommand\CR{\mathcal{R}}
\newcommand\CQ{\mathcal{Q}}

\newcommand\BCP{\mathbb{CP}}
\newcommand\BHP{\mathbb{HP}}
\newcommand\BR{\mathbb{R}}

\newcommand\BZ{\mathbb{Z}}

\newcommand\br{\bar}

\newcommand\bu{\mathbf{u}}
\newcommand\bv{\mathbf{v}}
\newcommand\be{\mathbf{e}}

\newcommand\bfZ{\mathbf{Z}}

\newcommand\scc{\kp}
\DeclareMathOperator{\Ric}{Ric}
\newcommand\scd{{\scc}_{1}}
\newcommand\sct{{\scc}_{2}}
\newcommand\Rcd{{\Ric}_{1}}
\newcommand\Rct{{\Ric}_{2}}

\newcommand\ST{S^{[2]}}

\newcommand\heat{\frac{\dd}{\dd t} - \Dt}

\newcommand\gstd{g_{\mathrm{std}}}
\newcommand\gFS{g_{\mathrm{FS}}}

\begin{document}

\title{A New Monotone Quantity in Mean Curvature Flow Implying Sharp Homotopic Criteria}

\author{Chung-Jun Tsai}
\address{Department of Mathematics, National Taiwan University, and National Center for Theoretical Sciences, Math Division, Taipei 10617, Taiwan}
\email{cjtsai@ntu.edu.tw}

\author{Mao-Pei Tsui}
\address{Department of Mathematics, National Taiwan University, and National Center for Theoretical Sciences, Math Division, Taipei 10617, Taiwan}
\email{maopei@math.ntu.edu.tw}

\author{Mu-Tao Wang}
\address{Department of Mathematics, Columbia University, New York, NY 10027, USA}
\email{mtwang@math.columbia.edu}


\thanks{C.-J.~Tsai is supported in part by the National Science and Technology Council grants 112-2636-M-002-003 and 112-2628-M-002-004-MY4.  M.-P.~Tsui is supported in part by the National Science and Technology Council grants 109-2115-M-002-006 and 112-2115-M-002-015-MY3. M.-T. ~Wang is supported in part by the National Science Foundation under Grants DMS-1810856 and DMS-2104212.  Part of this work was carried out when M.-T.~Wang was visiting the National Center of Theoretical Sciences.}

\begin{abstract}
A new monotone quantity in graphical mean curvature flows of higher codimensions is identified in this work. The submanifold deformed by the mean curvature flow is the graph of a map between Riemannian manifolds, and the quantity is monotone increasing under the area-decreasing condition of the map. The flow provides a natural homotopy of the corresponding map and leads to sharp criteria regarding the homotopic class of maps between complex projective spaces, and maps from spheres to complex projective spaces, among others. 
\end{abstract}

\maketitle

\section{Introduction}

We start with the following theorem and its corollary proved in \cite{TW04}.

\begin{thm}
Let $f: (S^n, \gstd) \to (S^m, \gstd) $ be an area-decreasing smooth map between standard unit spheres of dimensions $n, m\geq 2$. Let $\Gamma_f$ be the graph of $f$ as a submanifold of $S^n\times S^m$, then the mean curvature flow of $\Gamma_f$ in $S^n\times S^m$ exists smoothly and remains graphical for all time, and converges smoothly to the graph of a constant map.
\end{thm}

\begin{cor} \label{TW04}
Any area-decreasing map from $(S^n, \gstd)$ to  $(S^m, \gstd)$ with dimensions $n,m\geq2$ must be homotopically trivial. 
\end{cor}

Explicit examples of non-trivial homotopy classes in  $\pi_n(S^m)$ when $n\geq m$ inspire interest to find local geometric conditions that determine the homotopy class of a map from $S^n$ to  $S^m$. The ``stretch" or ``1-dilatation" of such a map has been considered by Lawson \cite{Lawson1968}, Hsu \cite{Hsu1972}, Gromov \cite{Gr78}, and DeTurck--Gluck--Storm \cite{DGS13}. In \cite{Gr96}, Gromov proposed to study how the $k$-dilation of a map $S^n \to S^m$ is related to its homotopy class. 
In particular, he showed in \cite{Gr96}*{page 179 and Corollary on page 230} (see also \cite{Gu07}*{Section 2}) that there exists a constant  $\vep(n,m) > 0$ such that a map $f: (S^n, \gstd) \to (S^m, \gstd) $ with $2$-dilation (see below) less than $\vep(n,m)$ everywhere must be homotopically trivial. It was conjectured that $\vep(n,m)$ can be chosen to be one, as a map with $1$-dilation less than one is clearly homotopically trivial. Unbeknown to the authors of \cite{TW04}, the conjecture was affirmed by Corollary \ref{TW04} : a map $f: (S^n, \gstd) \to (S^m, \gstd) $ with $2$-dilation less than one is exactly an area-decreasing map. By contrast, such a result does not hold for the $3$-dilation by the work of Guth \cite{Gu13}, see stronger results about $k$-dilations therein.

 Area-decreasing maps in the mean curvature flow arise in a rather different context. It was shown in \cite{W03} that the Gauss maps of a mean curvature flow form a harmonic map heat flow, and area-decreasing maps correspond to a convex region in the Grassmannian. It was proved in \cite{TW04} that the area-decreasing condition is preserved by the mean curvature flow under curvature assumptions of the domain and the target manifolds (see \cite{W05} for the case when the target space is 2-dimensional). There were several later extensions \cites{LL11, SS14, ASS22} of the results in \cite{TW04} to allow more general curvature conditions, all based on the study of the evolution equation of the symmetric 2-tensor $\ST$ (see \cite{TW04}*{section 5}, or section \ref{sec_ST} of this article), whose positivity corresponds to the area-decreasing condition. In this article, we consider the evolution equation  \eqref{detS} of the logarithmic determinant of $\ST$. This novel formulation significantly improves our understanding of the ambient curvature terms involved in the evolution equation which leads to the sharp homotopic criteria mentioned in the title. In particular, unlike previous works whose assumptions rely on the sectional curvatures of the domain and target manifolds, Theorem \ref{thm_Ric_curv} relies on the Ricci curvatures. This enables us to prove Corollaries \ref{cproj0} and \ref{cproj1} regarding the complex projective spaces. 

It is best to describe the area-decreasing condition and the new monotone quantity in terms of singular values. Recall for a $C^1$ map $f:(\Sigma_1, g_1)\to (\Sigma_2, g_2)$ between Riemannian manifolds, the singular values of $\dd f$ at a point $p\in \Sigma_1$ are the non-negative square roots of eigenvalues of $(\dd f|_p)^T(\dd f|_p)$. Suppose $\Sigma_1$ is of dimension $n$ and the singular values are $\lambda_i, i=1,\cdots, n$, then the $2$-dilation of $f$ is defined to be 
\begin{align} \label{2-dilation}
\max_{i<j}\lambda_i\lambda_j ~.
\end{align}
A map is said to be \emph{area-decreasing} if the $2$-dilation is less than one everywhere, or 
\begin{align} \label{area decreasing}
\lambda_i\lambda_j<1 \quad\text{for any } i<j ~.
\end{align}

This article considers the mean curvature flow of the graph of $f$, $\Gm_f$, as a submanifold of $\Sigma_1\times \Sigma_2$. Under suitable curvature assumptions on $\Sigma_1$ and $\Sigma_2$ and the area-decreasing condition, the flow remains the graph of a family of maps $f_t$. The new monotone quantity is \begin{equation}\label{logdetS^2}\log\left( \prod_{1\leq i<j\leq n}\frac{2(1-\ld_i^2\ld_j^2)}{(1+\ld_i^2)(1+\ld_j^2)}\right),\end{equation} where $\lambda_i$ are the singular values of $f_t$.

We state the main theorems in the following.

\begin{thm} \label{thm_sec_curv}
Let $(\Sm_1,g_1)$ and $(\Sm_2,g_2)$ be two compact, Riemannian manifolds.  Let $n = \dim\Sm_1$, $m = \dim\Sm_2$, and assume $n\geq m\geq2$.  Suppose that their sectional curvatures satisfy
\begin{align*}
\scc_{1} \geq 1 \quad\text{and}\quad \scc_{2}<\frac{2n-m-1}{m-1} ~.
\end{align*}
Then, for any area-decreasing map $f:\Sm_1\to\Sm_2$, the mean curvature flow in $\Sm_1\times\Sm_2$ starting from $\Gm_f$ exists for all time, and remains graphical.  Moreover, it converges smoothly to the graph of a constant map as $t\to\infty$.
\end{thm}
The condition $\scc_{1} \geq 1$ means that $\scc_{1}(\sm) \geq 1$ for any $\sm\in\Ld^2(T\Sm_1)$.

\begin{thm} \label{thm_Ric_curv}
Let $(\Sm_1,g_1)$ and $(\Sm_2,g_2)$ be two compact Riemannian manifolds.  Assume $\dim\Sm_1\geq\dim\Sm_2\geq2$.  Suppose that their Ricci and sectional curvatures satisfy
\begin{align*}
\frac{\Ric_{1}}{g_1} \geq \frac{\Ric_{2}}{g_2} \quad\text{and}\quad \scc_{1} + \scc_{2} > 0 ~.
\end{align*}
{When $\dim \Sm_1 > \dim\Sm_2$, we assume additionally that $\scc_{1} \geq 0$.}
Then, the same conclusion as Theorem \ref{thm_sec_curv} holds true.
\end{thm}
The first curvature condition means that $\Ric_{1}(\bu,\bu) \geq \Ric_{2}(\bv,\bv)$ for any unit vectors $\bu\in T\Sm_1$ and $\bv\in T\Sm_2$.  The second curvature condition means that $\scc_1(\sm_1) + \scc_2(\sm_2) > 0$ for any $\sm_1\in\Ld^2(T\Sm_1)$ and $\sm_2\in\Ld^2(T\Sm_2)$.

We state some corollaries of homotopic criteria which follow from the above theorems. Consider the complex projective spaces endowed with the Fubini--Study metrics, $\gFS$.  Our theorems imply the following results:
\begin{cor} \label{cproj0}
Suppose that $n\geq m\geq 1$. Any area-decreasing map from $(\BCP^n,\gFS)$ to $(\BCP^m,\gFS)$ must be homotopically trivial.
\end{cor}

 It is known that the homotopy class of a map in $[\BCP^n, \BCP^m]$ is determined by the induced map on the cohomology groups \cite{Mo84} for $n<m$, and by the degree for $n=m$ (see \cite{McGibbon82}). Since the identity map is not homotopically trivial, the area-decreasing condition in Corollary \ref{cproj0} is sharp.

Another corollary concerns $\pi_{2n+1}(\BCP^n)$, which is isomorphic to $\pi_{2n+1}(S^{2n+1})\cong\BZ$ by the long exact sequence induced by the Hopf fibration $S^1\to S^{2n+1}\to\BCP^n$.

\begin{cor} \label{cproj1}
Any smooth map $f: (S^{2n+1},\gstd) \to (\BCP^n,\gFS)$ with 2-dilation less than  $\frac{2n}{2n+1}$ everywhere must be homotopically trivial.
\end{cor}

The Hopf fibration $S^{2n+1}\to\BCP^n$ represents the generator of $\pi_{2n+1}(\BCP^n)\cong\BZ$.  It is a straightforward computation to show that the Hopf fibration $(S^{2n+1},\gstd)\to(\BCP^n,\gFS)$ has singular values $0$ (of multiplicity $1$) and $1$ (of multiplicity $2n$).  With this understood, Corollary \ref{cproj1} is sharp as $n\to\infty$.

The paper is organized as follows.  In section \ref{sec_pre}, we review some properties of graphical mean curvature flow.  In section \ref{sec_ST}, the evolution equation of the logarithmic determinant of $\ST$ is derived.  Section \ref{sec_mainthm} is devoted to prove the main theorems, Theorem \ref{thm_sec_curv} and \ref{thm_Ric_curv}.  In section \ref{sec_cor}, we explain the implications of the main theorems.

\begin{ack*}
    The authors thank Man-Chun Lee and Luen-Fai Tam for suggestions on the earlier draft of this paper.
\end{ack*}

\section{Preliminaries} \label{sec_pre}

For two Riemannian manifolds $(\Sm_1,g_1)$ and $(\Sm_2,g_2)$, consider the symmetric $2$-tensor $S$ on $\Sm_1\times\Sm_2$ defined by
\begin{align*}
S(X,Y) &= \ip{\pi_1(X)}{\pi_1(Y)} - \ip{\pi_2(X)}{\pi_2(Y)}
\end{align*} 
for any $X,Y\in T(\Sm_1\times\Sm_2)$. It is not hard to see that $S$ is parallel with respect to the Levi-Civita connection of the product metric.  Let $n=\dim\Sm_1$ and $m=\dim\Sm_2$.

For a map $f:\Sm_1\to\Sm_2$, let $\Gm_f$ be the graph of $f$.  As in \cite{TW04}*{section 4}, one can apply the singular value decomposition at any $p\in\Sm_1$ to $\dd f|_p$ to find orthonormal bases $\{\be_i\}_{i=1}^n$ for $T_{(p,f(p))}\Gm_f$, and $\{\be_\af\}_{\af = n+1}^{n+m}$ for $(T_{(p,f(p))}\Gm_f)^\perp$ such that
\begin{align}
S(\be_i,\be_j) &= \frac{1-\ld_i^2}{1+\ld_i^2}\dt_{ij} ~, &S(\be_i,\be_{\af}) &= \frac{-2\ld_i}{1+\ld_i^2}\dt_{i(\af-n)} ~, &S(\be_\af,\be_\bt) &= \frac{-(1-\ld_{\af-n}^2)}{1+\ld_{\af-n}^2}\dt_{\af\bt} ~,
\label{SVD0} \end{align}
where $\{\ld_i\}_{i=1}^n$ are the singular values of $\dd f|_p$.  Indeed, there exist orthonormal bases $\{\bu_i\}_{i=1}^n$ for $T_p\Sm_1$ and $\{\bv_j\}_{j=1}^m$ for $T_{f(p)}\Sm_2$ such that
\begin{align} \begin{split}
\be_i &= \frac{1}{\sqrt{1+\ld_i^2}}(\bu_i + \ld_i\bv_i)  \qquad\text{for }i\in\{1,\ldots,n\} ~, \\
\be_\af &= \frac{1}{\sqrt{1+\ld_{\af-n}^2}}(-\ld_{\af-n}\bu_{\af-n} + \bv_{\af-n})  \qquad\text{for }\af\in\{n+1,\ldots,n+m\} ~.
\end{split} \label{basis} \end{align}
It is convenient to set $\ld_i$ to be $0$ when $i>\min\{n,m\}$.

Suppose that $\{\Gm_{f_t}\}$ is a mean curvature flow.  Let $F_t$ be the embedding
\[x\in\Sm_1 \to (x,f_t(x))\in\Sm_1\times\Sm_2 ~.\]
Consider the tensor $F_t^*S$ on $\Sm_1$, and denote its components by $S_{ij}$.  From \cite{TW04}*{(3.7)}, 
$F_t^*S$ obeys the following equation along the mean curvature flow:
\begin{align} \begin{split}
(\heat)S_{ij} &= R_{kik\af}S_{\af j} + R_{kjk\af}S_{\af i} \\
&\quad + h_{\af k\ell}h_{\af ki}S_{\ell j} + h_{\af k\ell}h_{\af kj}S_{\ell i} - 2h_{\af ki}h_{\bt kj}S_{\af\bt} ~,
\end{split} \label{evol_S} \end{align}
which is in terms of an evolving orthonormal frame.  Here, the Laplacian $\Dt$ is the rough Laplacian for $2$-tensors of $\Gm_{f(t)}$, and $R_{kik\af} = \ip{R(\be_k,\be_\af)\be_i}{\be_k}$ is the coefficient of the Riemann curvature tensor of $\Sm_1\times\Sm_2$.

The (spatial) gradient of $F_t^*S$ is computed on \cite{TW04}*{p.1121} as follows:
\begin{align}
S_{ij;k} &= h_{\af ki}S_{\af j} + h_{\af kj}S_{\af i} ~,
\label{grad_S} \end{align}
where $h_{\af ki} = \ip{\br{\nabla}_{\be_k}\be_i}{\be_\af}$ is the coefficient of the second fundamental form of $\Gm_{f(t)}$ in $\Sigma_1\times \Sigma_2$.

In terms of the frame $\{\be_i\}_{i=1}^n$ \eqref{SVD0}, $F_t^*S$ is diagonal, and
\begin{align*}
(S_{ii} + S_{jj}) & = \frac{2(1-\ld_i^2\ld_j^2)}{(1+\ld_i^2)(1+\ld_j^2)} ~. 
\end{align*}
It follows that the $2$-positivity\footnote{The sum of any two eigenvalues is positive.} of $F_t^*S$ is equivalent to  $f(t)$ being area-decreasing.

\section{The Evolution Equation of $\log(\det\ST)$} \label{sec_ST}

In \cite{TW04}*{section 5}, the authors introduce a tensor $\ST$ from $F_t^*S$, which can be viewed as a symmetric endomorphism on $\Ld^2(T^*\Gm_f)$.  With respect to an orthonormal frame,
\begin{align}
\ST_{(ij)(k\ell)} &= S_{ik}\dt_{j\ell} + S_{j\ell}\dt_{ik} - S_{i\ell}\dt_{jk} - S_{jk}\dt_{i\ell} \label{def_S2}
\end{align}
for any $i<j$ and $k<\ell$.

Suppose that at a space-time point $p$, $S$ is diagonal, $S_{ij}|_p = S_{ii}\dt_{ij}$.  As in \eqref{SVD0}, at this point
\begin{align}
S_{ii} = \frac{1-\ld_i^2}{1+\ld_i^2} ~, ~\text{ and set }~ C_{ii} = \frac{2\ld_i}{1+\ld_i^2} ~.
\label{Cii} \end{align}
Note that $S_{ii}^2 + C_{ii}^2 = 1$.  The fact that $\ST$ is positive is equivalent to $S_{ii}+S_{jj}>0$ for $i< j$. We need the following calculation lemma for the evolution of $S_{ii}+S_{jj}$, $i< j$.

\begin{lem} \label{lem_Sij}
Suppose that $\ST$ is positive definite.  For any $i,j$ with $1\leq i<j\leq n$,
\begin{align} \begin{split}
&\quad (S_{ii}+S_{jj})^{-1}(\heat)(S_{ii}+S_{jj}) + \oh(S_{ii}+S_{jj})^{-2}|\nabla(S_{ii}+S_{jj})|^2 \\
&\geq 2\sum_{k}\left[h_{(n+j)ki}^2+h_{(n+i)kj}^2+h_{(n+i)ki}^2+h_{(n+j)kj}^2+\sum_{\ell>n}\left(h_{(n+\ell)ki}^2 + h_{(n+\ell)kj}^2\right) \right]\\
&\quad - 2(S_{ii}+S_{jj})^{-1}\sum_{k} \left[ R_{kik(n+i)}C_{ii} + R_{kjk(n+j)}C_{jj} \right] \\
&\quad + 2(S_{ii}+S_{jj})^{-2}\sum_k \left(h_{(n+i)ki}C_{jj} + h_{(n+j)kj}C_{ii}\right)^2 ~.
\end{split} \label{lndetSij0} \end{align}
In the above expression, the summation index $k$ is from $1$ to $n$. 
\end{lem}

\begin{proof}
By \eqref{evol_S} and \eqref{SVD0},
\begin{align*}
(\heat)S_{ii} &= 2\sum_{k,\af} R_{kik\af}S_{\af i} + 2\sum_{\af,k}h_{\af ki}^2S_{ii} - 2\sum_{\af,k}h_{\af ki}^2S_{\af\af} \\
&= -2\sum_{k} R_{kik(n+i)}C_{ii} + 2\sum_{k,\ell}h_{(n+\ell)ki}^2S_{ii} + 2\sum_{k,\ell}h_{(n+\ell)ki}^2S_{\ell\ell}
\end{align*}
where $k$ is from $1$ to $n$, and $\ell$ is from $1$ to $m$.
It follows that
\begin{align*}
(\heat)(S_{ii}+S_{jj}) &= -2\sum_{k} \left[ R_{kik(n+i)}C_{ii} + R_{kjk(n+j)}C_{jj} \right] \\
&\quad + 2\sum_{k,\ell}h_{(n+\ell)ki}^2(S_{ii}+S_{\ell\ell}) + 2\sum_{k,\ell}h_{(n+\ell)kj}^2(S_{jj}+S_{\ell\ell}) ~.
\end{align*}

On the other hand, by \eqref{grad_S} and \eqref{SVD0},
\begin{align}
|\nabla(S_{ii}+S_{jj})|^2 &= 4\sum_k\left(h_{(n+i)ki}C_{ii} + h_{(n+j)kj}C_{jj}\right)^2 ~. \label{nabla_Sij}
\end{align}

Therefore, 
\begin{align} \begin{split}
&\quad (S_{ii}+S_{jj})^{-1}(\heat)(S_{ii}+S_{jj}) + \oh(S_{ii}+S_{jj})^{-2}|\nabla(S_{ii}+S_{jj})|^2 \\
&=\quad - 2(S_{ii}+S_{jj})^{-1}\sum_{k} \left[ R_{kik(n+i)}C_{ii} + R_{kjk(n+j)}C_{jj} \right]+2(S_{ii}+S_{jj})^{-1}\left[ \mbox{I}+\mbox{II}\right] ~,
\end{split} \label{lndetSij} \end{align}
where 
\begin{align*}
    \mbox{I} &= \sum_{k,\ell}h_{(n+\ell)ki}^2(S_{ii}+S_{\ell\ell}) + \sum_{k,\ell}h_{(n+\ell)kj}^2(S_{jj}+S_{\ell\ell}) ~,
\end{align*}
and
\begin{align*}
    \mbox{II} &=(S_{ii}+S_{jj})^{-1}\sum_k\left[h_{(n+i)ki}C_{ii} + h_{(n+j)kj}C_{jj}\right]^2 ~.
\end{align*}

We claim that 
\begin{align}\begin{split}
\mbox{I}+\mbox{II} &\geq  (S_{ii}+S_{jj})\sum_{k}\left[h_{(n+j)ki}^2+h_{(n+i)kj}^2+h_{(n+i)ki}^2+h_{(n+j)kj}^2+\sum_{\ell>n}\left(h_{(n+\ell)ki}^2 + h_{(n+\ell)kj}^2\right) \right]\\
&\quad + (S_{ii}+S_{jj})^{-1}\sum_k \left(h_{(n+i)ki}C_{jj} + h_{(n+j)kj}C_{ii}\right)^2,  \end{split} \label{I+II}\end{align} which proves the lemma. 

By taking out $\ell=i,j$, the term $\mbox{I}$ can be rewritten as
\begin{align}
\begin{split}\mbox{I} &= \mbox{I}'  + \sum_k\left[h_{(n+j)ki}^2 + h_{(n+i)kj}^2\right](S_{ii} + S_{jj}) \\
&\quad + \sum_k\sum_{\ell\neq\{i,j\}}h_{(n+\ell)ki}^2(S_{ii}+S_{\ell\ell}) + \sum_k\sum_{\ell\neq\{i,j\}}h_{(n+\ell)kj}^2(S_{jj}+S_{\ell\ell}) ~, \end{split} \label{Sij1}
\end{align} where
\begin{align*}
    \mbox{I}' &= 2\sum_{k} \left[h_{(n+i)ki}^2S_{ii} + h_{(n+j)kj}^2S_{jj} \right] ~.
\end{align*}

If $m>n$, consider the $\ell>n$ terms in the last two summands (both are non-negative due to our assumption)  in \eqref{Sij1}.  When $\ell>n$, $S_{\ell\ell} = 1$.  It follows that
\[ S_{ii}+S_{\ell\ell}\geq S_{ii}+S_{jj} \quad\text{and}\quad S_{jj}+S_{\ell\ell}\geq S_{ii}+S_{jj} ~. \]
Then
\begin{align*}
&\quad \sum_k\sum_{\ell>n}h_{(n+\ell)ki}^2(S_{ii}+S_{\ell\ell}) + \sum_k\sum_{\ell>n}h_{(n+\ell)kj}^2(S_{jj}+S_{\ell\ell}) \\
&\geq (S_{ii} + S_{jj})\sum_k\sum_{\ell>n} \left[h_{(n+\ell)ki}^2 + h_{(n+\ell)kj}^2\right] ~.
\end{align*}

Therefore, 
\begin{align*}
\mbox{I} &\geq \mbox{I}'  + (S_{ii} + S_{jj}) \sum_k\left[h_{(n+j)ki}^2 + h_{(n+i)kj}^2+\sum_{\ell>n} (h_{(n+\ell)ki}^2 + h_{(n+\ell)kj}^2)\right] ~.
\end{align*}

Grouping $\mbox{I}'$ and $\mbox{II}$ yields
\begin{align*}
&\quad \mbox{I}'+\mbox{II}\\
&= 2\sum_{k} \left[h_{(n+i)ki}^2S_{ii} + h_{(n+j)kj}^2S_{jj}\right] +(S_{ii}+S_{jj})^{-1}\sum_k\left[h_{(n+i)ki}C_{ii} + h_{(n+j)kj}C_{jj}\right]^2\\
&= (S_{ii}+S_{jj}) \sum_{k}\left[h_{(n+i)ki}^2 + h_{(n+j)kj}^2 \right] + (S_{ii}+S_{jj})^{-1}\sum_k\left[h_{(n+i)ki}C_{jj} + h_{(n+j)kj}C_{ii}\right]^2,
\end{align*}
where the last equality follows from the identity
\begin{align*}
    2S_{ii}+(S_{ii}+S_{jj})^{-1}C_{ii}^2= (S_{ii}+S_{jj})+(S_{ii}+S_{jj})^{-1} C_{jj}^2 ~. 
\end{align*}
Putting these together proves claim \eqref{I+II}.
\end{proof}


We are now ready to derive the evolution equation of $\log(\det\ST)$.

\begin{thm} \label{thm_detS}
Suppose that $\ST$ is positive definite.  The function $\log(\det\ST)$ satisfies
\begin{align}
(\heat)\log(\det\ST) &\geq 2|A|^2 + 2(n-2)\sum_{1\leq i\leq n}\sum_{1\leq k\leq n} h_{(n+i)ki}^2 + 2\CR_S + 2\CQ_S
\label{detS} \end{align}
where
\begin{align}
\CR_S = & - \sum_{1\leq i<j\leq n} (S_{ii}+S_{jj})^{-1}\sum_{1\leq k\leq n} \left[ R_{kik(n+i)}C_{ii} + R_{kjk(n+j)}C_{jj} \right]
\label{RS} \end{align}
and
\begin{align} \begin{split}
\CQ_S = & \sum_{1\leq i<j\leq n}(S_{ii}+S_{jj})^{-2}\sum_{1\leq k\leq n} \left(h_{(n+i)ki}C_{ii} + h_{(n+j)kj}C_{jj}\right)^2 \\
+& \sum_{1\leq i<j\leq n}(S_{ii}+S_{jj})^{-2}\sum_{1\leq k\leq n} \left(h_{(n+i)ki}C_{jj} + h_{(n+j)kj}C_{ii}\right)^2  ~.
\end{split} \label{QS} \end{align}
In particular, $\CQ_S\geq0$.
\end{thm}

\begin{proof}
Suppose that $\ST$ is positive definite, and denote its inverse by $Q$.  Note that $Q^{(ij)(k\ell)}|_p = (S_{ii} + S_{jj})^{-1}\dt^{ik}\dt^{j\ell}$.  Consider the function $\log \det(\ST)$.  At $p$, this quantity is
\[ \sum_{1\leq i<j\leq n}\log(S_{ii}+S_{jj}) ~. \]
For brevity, write $A,B$ for the indices of $\ST$.  We compute
\begin{align}
(\heat)\log(\det \ST)
&= Q^{AB}\left[(\frac{\dd}{\dd t}-\Dt)\ST_{AB}\right] - (\nabla_\ell Q^{AB})(\nabla_\ell \ST_{AB}) \notag \\
&= Q^{AB}\left[(\frac{\dd}{\dd t}-\Dt)\ST_{AB}\right] + Q^{AC}(\nabla_\ell\ST_{CD})Q^{DB}(\nabla_\ell\ST_{AB}) \notag \\
&\stackrel{\text{at }p}{=} \sum_A Q^{AA}\left[(\frac{\dd}{\dd t}-\Dt)\ST_{AA}\right] + \sum_{A,B} Q^{AA}Q^{BB} \left|\nabla\ST_{AB}\right|^2 ~. \label{lndet01}
\end{align}
The first term on the last line of \eqref{lndet01} is
\begin{align}
\sum_{1\leq i<j\leq n}(S_{ii}+S_{jj})^{-1}(\frac{\dd}{\dd t}-\Dt)(S_{ii}+S_{jj}) ~. \label{lndet02}
\end{align}
For the second term, it follows from definition \eqref{def_S2} that $\ST_{(ij)(k\ell)} \equiv 0$ when $i,j,k,\ell$ are all distinct.  If $i,j,k$ are all distinct, $\ST_{AB} = \pm S_{k\ell}$ for $A = (ij)$ or $(ji)$, and $B = (ik)$ or $(ki)$.  With this understood, the second term of \eqref{lndet01} becomes
\begin{align}
\sum_{1\leq i<j\leq n} (S_{ii}+S_{jj})^{-2}|\nabla(S_{ii}+S_{jj})|^2 + 2\sum_{1\leq i\leq n}\sum_{\substack{1\leq j<k\leq n\\j\neq i,\,k\neq i}}(S_{ii}+S_{jj})^{-1}(S_{ii}+S_{kk})^{-1}|\nabla{S_{jk}}|^2 ~. \label{lndet03}
\end{align}
Since the last summation is taken over $j<k$, there is a coefficient $2$.

With \eqref{lndet01}, \eqref{lndet02} and \eqref{lndet03}, one finds that
\begin{align*}
&\quad (\heat)\log(\det\ST) \\
&\geq \sum_{1\leq i<j\leq n} \left[ (S_{ii}+S_{jj})^{-1}(\heat)(S_{ii}+S_{jj}) + \oh(S_{ii}+S_{jj})^{-2}|\nabla(S_{ii}+S_{jj})|^2 \right] \\
&\quad + \oh\sum_{1\leq i<j\leq n} (S_{ii}+S_{jj})^{-2}|\nabla(S_{ii}+S_{jj})|^2 ~. 
\end{align*}

By invoking Lemma \ref{lem_Sij} and \eqref{nabla_Sij},
\begin{align*} 
&\quad (\heat)\log(\det(\ST)) - 2\CR_S - 2\CQ_S \\
&\geq 2\sum_{1\leq i<j\leq n}\sum_{1\leq k\leq n}\left[h_{(n+j)ki}^2+h_{(n+i)kj}^2+h_{(n+i)ki}^2+h_{(n+j)kj}^2 + 2\sum_{\ell>n}\left(h_{(n+\ell)ki}^2 + h_{(n+\ell)kj}^2\right) \right] \\
&\geq 2|A|^2 + 2(n-2)\sum_{1\leq i\leq n}\sum_{1\leq k\leq n} h_{(n+i)ki}^2 ~.
\end{align*}
This finishes the proof of this theorem.
\end{proof}

\subsection{Some Estimates}

Consider the function
\begin{align}
\vph(\ld_1,\ldots,\ld_n) &= \log\left(\prod_{1\leq i<j\leq n}\frac{1-\ld_i^2\ld_j^2}{(1+\ld_i^2)(1+\ld_j^2)}\right)
\label{fct_ld} \end{align}
defined where $\ld_i^2\ld_j^2 < 1$ for any $i\neq j$.  Since
\[0 < \frac{1-\ld_i^2\ld_j^2}{(1+\ld_i^2)(1+\ld_j^2)} \leq1 ~,\]
$\vph$ always takes value within $(-\infty,0]$.  It turns out that one can conclude the upper bound of $\ld_i^2$ from the lower bound of $\vph$.

\begin{lem} \label{est_vph}
Consider the function $\vph(\ld_1,\ldots,\ld_n)$ defined by \eqref{fct_ld} on $\{ (\ld_1,\ldots,\ld_n)\in\BR^n \,:\, \ld_i^2\ld_j^2 < 1 \text{ for any }i\neq j\}$.  Suppose that $\vph(\ld_1,\ldots,\ld_n) \geq -\dt$ for some $\dt>0$.  Then,
\begin{enumerate}
\item $\ld_i^2 \leq e^{\dt} - 1$ for all $i$;
\item $\ld_i^2\ld_j^2 \leq {(e^\dt-1)}/{(e^\dt+1)}$ for any $i\neq j$;
\item there exists $c_1 = c_1(n,\dt)>0$ such that $\left|\vph(\ld_1,\ldots,\ld_n)\right| \leq c_1\,\sum_{i=1}^{n}\ld_i^2$.
\end{enumerate}
\end{lem}

\begin{proof}
It follows from $\vph\geq -\dt$ that
\begin{align*}
\prod_{1\leq i<j\leq n}\frac{1-\ld_i^2\ld_j^2}{(1+\ld_i^2)(1+\ld_j^2)} &\geq e^{-\dt} ~.
\end{align*}
Since $0 < 1-\ld_i^2\ld_j^2 \leq 1$ and $1+\ld_i^2\geq 1$ for every $i\neq j$, one finds that $e^\dt(1 - \ld_i^2\ld_j^2) \geq (1+\ld_i^2)(1+\ld_j^2)$ for every $i\neq j$.  Hence,
\begin{align}
e^\dt - 1 \geq \ld_i^2 + \ld_j^2 \quad\text{and}\quad
\frac{e^\dt-1}{e^\dt+1} \geq \ld_i^2\ld_j^2 ~.
\label{est_vph1} \end{align}

Since $-\log(1-x) \leq \frac{x}{1-x}$ for $0\leq x<1$ and $\log(1+x)\leq x$ for $x>-1$,
\begin{align*}
-\vph(\ld_1,\ldots,\ld_n) &= \sum_{1\leq i<j\leq n} \left[ - \log(1-\ld_i^2\ld_j^2) + \log(1+\ld_i^2) + \log(1+\ld_j^2) \right] \\
&\leq \sum_{1\leq i<j\leq n} \left[ \frac{\ld_i^2\ld_j^2}{1-\ld_i^2\ld_j^2} + \ld_i^2 + \ld_j^2 \right] \\
&\leq c_1(n,\dt)\,\sum_{i=1}^n\ld_i^2
\end{align*}
where the last inequality uses \eqref{est_vph1}.  This finishes the proof of the lemma.
\end{proof}

For an area-decreasing map $f$, $\ST$ on $\Gm_f$ is positive definite.  In terms of the singular values \eqref{SVD0},
\begin{align*}
\log(\det\ST) 
&= \frac{n(n-1)}{2}\log 2 + \log\left( \prod_{1\leq i<j\leq n}\frac{2(1-\ld_i^2\ld_j^2)}{(1+\ld_i^2)(1+\ld_j^2)}\right) ~.
\end{align*}
The above discussion can be used to bound the gradient of $\log(\det\ST)$.

\begin{lem} \label{lem_grad_ST}
There exists a constant $c_2 = c_2(n) > 0$ with the following significance.  For an area-decreasing map $f:\Sm_1\to\Sm_2$, suppose that at a point $p\in\Sm_1$ there exists a $\dt > 0$ such that
\begin{align*}
\log(\det\ST) -  \frac{n(n-1)}{2}\log 2 &\geq -\dt  \quad\text{at }p ~.
\end{align*}
Then, $\left| \nabla\log(\det\ST) \right|^2 \leq c_2\,e^{4\dt}(e^\dt-1) \, |A|^2$ at $p$.
\end{lem}

\begin{proof}
Since $\nabla\det\ST = (\det\ST)(Q^{AB}\nabla\ST_{AB})$ and $Q^{AB}|_p = (S_{ii}+S_{jj})^{-1}$ for $A=B=(ij)$, we have
\begin{align*}
\nabla_k\log(\det\ST) &\stackrel{\text{at }p}{=}  \sum_{1\leq i<j\leq n}(S_{ii}+S_{jj})^{-1}(S_{ii;k}+S_{jj;k}) \\
&= -2\sum_{1\leq i<j\leq n} \frac{(1+\ld_i^2)(1+\ld_j^2)}{(1-\ld_i^2\ld_j^2)} \left( h_{(n+i)ki}\frac{\ld_i}{1+\ld_i^2} + h_{(n+j)kj} \frac{\ld_j}{1+\ld_j^2} \right) ~.
\end{align*}
Together with Lemma \ref{est_vph} (i) and (ii), this implies that
\begin{align*}
\left| \nabla_k\log(\det\ST) \right| &\leq n(n-1)\, e^\dt(1+e^\dt)\sqrt{e^\dt-1} \, |A| ~.
\end{align*}
The lemma follows from this estimate.
\end{proof}

\section{Proof of Main Theorems} \label{sec_mainthm}

This section is devoted to prove Theorems \ref{thm_sec_curv} and \ref{thm_Ric_curv}.
If $f_t$ is area-decreasing, then Theorem \ref{thm_detS} implies that
\begin{align}
(\heat)\log(\det\ST) &\geq 2|A|^2 + 2\CR_S ~.
\label{detS1} \end{align}
We are going to express $\CR_S$ \eqref{RS} in terms of the curvatures of $g_1$ and $g_2$, and the singular values of $f_t$.  Due to \eqref{basis}, the ambient curvature $R_{kik(n+i)}$ is
\begin{align}
\ip{R(\be_k,\be_{n+i})\be_i}{\be_k} &= \ip{R_{g_1}(\pi_1(\be_k),\pi_1(\be_{n+k}))(\pi_1(\be_i))}{\pi_1(\be_k)} \notag \\
&\quad + \ip{R_{g_2}(\pi_2(\be_k),\pi_2(\be_{n+k}))(\pi_2(\be_i))}{\pi_2(\be_k)} \notag \\
&= -\frac{\ld_i}{1+\ld_i^2} \left( \frac{1}{1+\ld_k^2}\,\scd(i,k) - \frac{\ld_k^2}{1+\ld_k^2}\,\sct(i,k) \right) \label{curv_notation}
\end{align}
where $\scd(i,k)$ is the sectional curvature of $\bu_i\w\bu_k$, and $\sct(i,k)$ is the sectional curvature of $\bv_{i}\w\bv_{k}$.  With \eqref{Cii},
\begin{align} \begin{split}
\CR_S &= \sum_{1\leq i<j\leq n} (S_{ii}+S_{jj})^{-1} \left[ \frac{2\ld_i^2}{(1+\ld_i^2)^2} \sum_{1\leq k\leq n}\left( \frac{1}{1+\ld_k^2}\,\scd(i,k) - \frac{\ld_k^2}{1+\ld_k^2}\,\sct(i,k) \right) \right. \\
&\qquad\qquad\qquad\qquad\qquad \left. + \frac{2\ld_j^2}{(1+\ld_j^2)^2} \sum_{1\leq k\leq n}\left( \frac{1}{1+\ld_k^2}\,\scd(j,k) - \frac{\ld_k^2}{1+\ld_k^2}\,\sct(j,k) \right) \right]
\end{split} \label{RS0} \end{align}


\subsection{Proof of Theorem \ref{thm_sec_curv}}

\subsubsection*{Re-group the curvature term}
Suppose that $n\geq m\geq 2$, and
\begin{align*}
\scd\geq1 \qquad\text{and}\qquad \sct\leq\tau
\end{align*}
for some $\tau>0$.
It follows that the curvature term \eqref{RS0} in \eqref{detS1} obeys
\begin{align*}
\CR_S &\geq \sum_{i<j}(S_{ii}+S_{jj})^{-1} \left[ \frac{\ld_i^2}{(1+\ld_i^2)^2} P_i  + \frac{\ld_j^2}{(1+\ld_j^2)^2} P_j\right] ~,
\end{align*} where 
\begin{align*}
P_i &= 2\sum_{1\leq k\leq n, k\not=i}\left( \frac{1}{1+\ld_k^2} - \frac{\ld_k^2}{1+\ld_k^2}\tau \right) ~.
\end{align*}

Note that $P_i$ can be rewritten as 
\begin{align*}
P_i &= 2\sum_{1\leq k\leq n, k\not=i}\left( (1+S_{kk}) - (1-S_{kk})\tau \right) ~.
\end{align*}
Since $S_{kk} = 1$ for $k\geq \min\{n,m\} = m$, one finds that
\begin{align*}
P_i &= (2n-m-1)-(m-1)\tau+(1+\tau)\left(-S_{ii} + \sum_{1\leq k\leq m}S_{kk}\right) ~.
\end{align*}
When $m>2$, we have $-S_{ii} + \sum_{1\leq k\leq m}S_{kk}>0$, and $P_i\geq (2n-m-1)-(m-1)\tau$.
When $m=2$, we claim that 
$$\sum_{i<j}(S_{ii}+S_{jj})^{-1} \left[ \frac{\ld_i^2}{(1+\ld_i^2)^2} (-S_{ii} + \sum_{1\leq k\leq m}S_{kk}) + \frac{\ld_j^2}{(1+\ld_j^2)^2} (-S_{jj} + \sum_{1\leq k\leq m}S_{kk}) \right] >0$$ 
In either case, we have 
\begin{align}
\CR_S &\geq \sum_{1\leq i<j\leq n} (S_{ii} + S_{jj})^{-1} \left[ \frac{\ld_i^2}{(1+\ld_i^2)^2} + \frac{\ld_j^2}{(1+\ld_j^2)^2} \right] \big( (2n-m-1) - (m-1)\tau \big)~. \label{RS1}
\end{align}

The proof of the claim is based on direct computations.  When $m = 2$,
\begin{align*}
&\quad \sum_{1\leq i<j\leq 2} (S_{ii} + S_{jj})^{-1} \left[ \frac{\ld_i^2}{(1+\ld_i^2)^2}\left( -S_{ii} + \sum_{1\leq k\leq 2}S_{kk} \right) + \frac{\ld_j^2}{(1+\ld_j^2)^2}\left( -S_{jj} + \sum_{1\leq k\leq 2}S_{kk} \right) \right] \\
&= (S_{11}+S_{22})^{-1} \left[ \frac{\ld_1^2}{(1+\ld_1^2)^2}\frac{1-\ld_2^2}{1+\ld_2^2} + \frac{\ld_2^2}{(1+\ld_2^2)^2}\frac{1-\ld_1^2}{1+\ld_1^2} \right] \\
&= (S_{11}+S_{22})^{-1} \frac{(\ld_1^2 + \ld_2^2)(1-\ld_1^2\ld_2^2)}{(1+\ld_1^2)^2(1+\ld_2^2)^2} \geq 0 ~,
\end{align*}
and
\begin{align*}
&\sum_{1\leq i\leq 2} (S_{ii} + 1)^{-1} \frac{\ld_i^2}{(1+\ld_i^2)^2}\left( -S_{ii} + \sum_{1\leq k\leq 2}S_{kk} \right)\\
&= \oh\left( \frac{\ld_1^2}{1+\ld_1^2}\frac{1-\ld_2^2}{1+\ld_2^2} + \frac{\ld_2^2}{1+\ld_2^2}\frac{1-\ld_1^2}{1+\ld_1^2} \right) \\
&= \frac{ (\ld_1 - \ld_2)^2+2\ld_1\ld_2(1-\ld_1\ld_2)}{2(1+\ld_1^2)(1+\ld_2^2)} \geq 0 ~.
\end{align*}

Let
\begin{align*}
\Phi &= \log\det(\ST) - \frac{n(n-1)}{2}\log2 ~,
\end{align*}
which is the $\log\det$ of $(\oh F_t^*S)^{[2]}$.  In terms of singular values, $\Phi$ is equal to \eqref{fct_ld}.
According to \eqref{detS1} and \eqref{RS1},
\begin{align} \begin{split}
&\quad (\heat)\Phi - 2|A|^2 \\
&\geq \sum_{1\leq i<j\leq n} (S_{ii} + S_{jj})^{-1} \left[ \frac{\ld_i^2}{(1+\ld_i^2)^2} + \frac{\ld_j^2}{(1+\ld_j^2)^2} \right] \left( (2n-m-1) - (m-1)\tau \right)
\end{split} \label{detS2} \end{align}
as long as $f_t$ is area-decreasing.

\subsubsection*{Long-Time Existence}

The assumption of Theorem \ref{thm_sec_curv} implies that $\tau$ can be chosen to be less than $\frac{2n-m-1}{m-1}$, and $\Phi \geq -\dt_0$ for some $\dt_0 > 0$ at $t=0$.

As long as $f_t$ is area-decreasing, it follows from \eqref{detS2} that
\begin{align}
(\heat)\Phi &\geq 2|A|^2\geq 0 ~.  
\label{detS3} \end{align}
Let
\begin{align*}
T_0 &= \sup\{ T\geq0 : \Gm_{f_t}\text{ exists on } [0,T) \text{ and } f_t \text{ is area-decreasing on }[0,T)\} ~.
\end{align*}
By the maximum principle on \eqref{detS3}, $\min_{\Gm_{f_t}}\Phi$ is non-decreasing in $t$, and thus $\Phi \geq -\dt_0$ for any $t\in[0,T_0)$, it follows from Lemma \ref{est_vph} (i) and (ii) that being graphical and being area-decreasing is preserved for $t\in[0,T_0)$.  Moreover, the quantitative bounds, Lemma \ref{est_vph} (i) and (ii), imply that $T_0$ must be the maximal existence time of the mean curvature flow.

Thus, \eqref{detS3} holds as long as the flow exists.  With such an inequality, one can perform the blow-up argument on Huisken's backward heat kernel, and apply White's regularity theorem to conclude that there is no finite time singularity.  This argument is the same as \cite{W02}*{Theorem A}; see also \cite{LL11}*{p.5751}.  It will be omitted here.

\subsubsection*{Smooth Convergence to a Constant Map}

The first task is to study $\lim_{t\to\infty}\Phi$.  Since $\Phi \geq -\dt_0$ for all $t$, Lemma \ref{est_vph} (i) and (ii) imply that there exists $c_3 = c_3(n,m,\tau,\dt_0) > 0$ such that the right-hand side of \eqref{detS2} is no less than $c_3\,\sum_{i=1}^n\ld_i^2$.  Due to Lemma \ref{est_vph} (iii), there exists $c_1 = c_1(n,\dt_0)$ such that \eqref{detS2} becomes
\begin{align}
(\heat)\Phi &\geq 2|A|^2 - \frac{c_3}{c_1}\Phi ~.
\label{detS4} \end{align}
By the maximum principle, $0\geq \Phi\geq -\dt_0\exp(-\frac{c_3}{c_1}t)$ on $\Gm_{f_t}$ for any $t\geq0$.  Therefore, $\Phi\to0$ as $t\to\infty$.  With Lemma \ref{est_vph} (i), one finds that $\ld_i\to0$ as $t\to\infty$.  In other words, $f_t$ converges to a constant map in the level of first order derivative.

The next step is to derive the equation for $\exp(\Phi)$, which can be used to bound the second fundamental form.  Recall that $\Phi$ and $\log(\det\ST)$ differ by a constant, and hence $\nabla\Phi = \nabla\log(\det\ST)$.  Let $\dt(t) = \dt_0\exp(-\frac{c_3}{c_1}t)$.  According to Lemma \ref{lem_grad_ST},
\begin{align*}
\left|\nabla\log(\det\ST)\right|^2 &\leq c_2\, e^{4\dt(t)}(e^{\dt(t)}-1)\,|A|^2
\end{align*}
on $\Gm_{f_t}$ for any $t\geq0$.  Therefore, there exists a $T_1\geq1$ such that $\left|\nabla\log(\det\ST)\right|^2 \leq |A|^2$ whenever $t\geq T_1$.  With \eqref{detS3},
\begin{align*}
(\heat)\exp(\Phi) &= \exp(\Phi) \left[ (\heat)\Phi - \left|\nabla\log(\det\ST)\right|^2 \right] \\
&\geq \exp(\Phi)\,|A|^2
\end{align*}
whenever $t\geq T_1$.  By the same argument as in \cite{LL11}*{pp.5753-5754}, this inequality can be used to prove that $\max_{\Gm_t}|A|^2\to0$ as $t\to\infty$.  With $\ld_i\to0$ as $t\to\infty$, $f_t\to\text{constant map}$ as $t\to\infty$.  This finishes the proof of Theorem \ref{thm_sec_curv}.

\subsection{Proof of Theorem \ref{thm_Ric_curv}} \label{sec_proof_2}

The proof of Theorem \ref{thm_Ric_curv} is completely parallel to that of Theorem \ref{thm_sec_curv}.  We only emphasize how to derive equations analogous to \eqref{detS2}, \eqref{detS3} and \eqref{detS4}.

Similar to \eqref{curv_notation}, write $\Rcd(i,j)$ for the Ricci curvature of $(\bu_i,\bu_j)$, and $\Rct(i,j)$ for the Ricci curvature of $(\bv_i,\bv_j)$, where $\{\bu_i\}$ and $\{\bv_i\}$ are the bases given by \eqref{basis}.
When $k > m$, set $\sct(k,i) = 0$ and $\Rct(k,i) = 0$ for any $i$.  

Rewrite \eqref{RS0} as follows.
\begin{align*}
&\quad \frac{2\ld_i^2}{(1+\ld_i^2)^2} \sum_{1\leq k\leq n}\left( \frac{1}{1+\ld_k^2}\,\scd(i,k) - \frac{\ld_k^2}{1+\ld_k^2}\,\sct(i,k) \right) \\
&= \frac{\ld_i^2}{(1+\ld_i^2)^2} \sum_{1\leq k\leq n} \big[ (1+S_{kk})\,\scd(i,k) + (-1+S_{kk})\,\sct(i,k) \big] \\
&= \frac{\ld_i^2}{(1+\ld_i^2)^2} \left[ \left(\Rcd(i,i) - \Rct(i,i)\right) + \sum_{1\leq k\leq n}S_{kk}(\scd(i,k) + \sct(i,k)) \right] ~.
\end{align*}
Hence,
\begin{align}
\begin{split} \CR_S &= \sum_{1\leq i<j\leq n}(S_{ii}+S_{jj})^{-1} \left[ \frac{\ld_i^2}{(1+\ld_i^2)^2}\left(\Rcd(i,i) - \Rct(i,i)\right) \right. \\
&\qquad\qquad\qquad\qquad\qquad\quad \left. + \frac{\ld_j^2}{(1+\ld_j^2)^2}\left(\Rcd(j,j) - \Rct(j,j)\right) \right] \end{split} \label{Rica} \\
\begin{split} &\quad + \sum_{1\leq i<j\leq n}(S_{ii} + S_{jj})^{-1} \left[ \frac{\ld_i^2}{(1+\ld_i^2)^2}\sum_{1\leq k\leq n}S_{kk}(\scd(i,k) + \sct(i,k)) \right. \\
&\qquad\qquad\qquad\qquad\qquad\qquad \left. + \frac{\ld_j^2}{(1+\ld_j^2)^2}\sum_{1\leq k\leq n}S_{kk}(\scd(j,k) + \sct(j,k)) \right] ~. \end{split} \label{Ricb}
\end{align}

Re-group \eqref{Ricb} as
\begin{align*}
\sum_{1\leq i<j\leq n}\sum_{1\leq k\leq n} W_{ijk} &= \sum_{1\leq i<j\leq n}(W_{iji} + W_{ijj}) + \sum_{1\leq i<j\leq n}\sum_{\substack{1\leq k\leq n\\ k\neq\{i,j\}}} W_{ijk} \\
&= \sum_{1\leq i<j\leq n}(W_{iji} + W_{ijj}) + \sum_{1\leq i<j<k \leq n}(W_{ijk} + W_{jki} + W_{ikj}) ~,
\end{align*}
where $W_{ijk}$ is
\begin{align*}
(S_{ii} + S_{jj})^{-1} \left[ \frac{\ld_i^2}{(1+\ld_i^2)^2}S_{kk}(\scd(i,k) + \sct(i,k)) + \frac{\ld_j^2}{(1+\ld_j^2)^2}S_{kk}(\scd(j,k) + \sct(j,k)) \right] ~.
\end{align*}
We compute
\begin{align*}
&\quad W_{iji} + W_{ijj} \\
&= (S_{ii} + S_{jj})^{-1} \left[ \frac{\ld_j^2}{(1+\ld_j^2)^2}S_{ii}(\scd(j,i) + \sct(j,i)) + \frac{\ld_i^2}{(1+\ld_i^2)^2}S_{jj}(\scd(i,j) + \sct(i,j))\right] \\
&= \frac{\ld_i^2+\ld_j^2}{2(1+\ld_i^2)(1+\ld_j^2)}(\scd(i,j) + \sct(i,j)) ~,
\end{align*}
and
\begin{align*}
&\quad W_{ijk} + W_{jki} + W_{ikj} \\
&= (S_{ii} + S_{jj})^{-1} \left[ \frac{\ld_i^2}{(1+\ld_i^2)^2}S_{kk}(\scd(i,k) + \sct(i,k)) + \frac{\ld_j^2}{(1+\ld_j^2)^2}S_{kk}(\scd(j,k) + \sct(j,k)) \right] \\
&\quad + (S_{jj} + S_{kk})^{-1} \left[ \frac{\ld_j^2}{(1+\ld_j^2)^2}S_{ii}(\scd(j,i) + \sct(j,i)) + \frac{\ld_k^2}{(1+\ld_k^2)^2}S_{ii}(\scd(k,i) + \sct(k,i)) \right] \\
&\quad + (S_{ii} + S_{kk})^{-1} \left[ \frac{\ld_i^2}{(1+\ld_i^2)^2}S_{jj}(\scd(i,j) + \sct(i,j)) + \frac{\ld_k^2}{(1+\ld_k^2)^2}S_{jj}(\scd(k,j) + \sct(k,j)) \right] \\
&= V_{ijk}(\scd(i,j) + \sct(i,j)) + V_{jki}(\scd(j,k) + \sct(j,k)) + V_{ikj}(\scd(i,k) + \sct(i,k))
\end{align*}
where
\begin{align}
V_{ijk} &= \frac{(1+\ld_k^2)(\ld_i^2+\ld_j^2-2\ld_i^2\ld_j^2-2\ld_i^2\ld_j^2\ld_k^2 + \ld_i^4\ld_j^2\ld_k^2 + \ld_i^2\ld_j^4\ld_k^2)}{2(1+\ld_i^2)(1+\ld_j^2)(1-\ld_i^2\ld_k^2)(1-\ld_j^2\ld_k^2)} \notag \\
\begin{split} &= \frac{(1+\ld_k^2)}{2(1+\ld_i^2)(1+\ld_j^2)(1-\ld_i^2\ld_k^2)(1-\ld_j^2\ld_k^2)} \left[ (\ld_i-\ld_j)^2(1+\ld_i^2\ld_j^2\ld_k^2) \right.\\
&\qquad\qquad\qquad\qquad\qquad\qquad\qquad\qquad\qquad\quad \left. + 2\ld_i\ld_j(1-\ld_i\ld_j\ld_k^2)(1-\ld_i\ld_j) \right] ~.
\end{split} \label{fct_V} \end{align}
In particular, $V_{ijk}\geq0$ provided that $f_t$ is area-decreasing.

It follows that under the assumption of Theorem \ref{thm_Ric_curv},
\begin{align*}
\CR_S &= \sum_{1\leq i<j\leq n} \left[ \frac{\ld_i^2(1+\ld_j^2)}{2(1+\ld_i^2)(1-\ld_i^2\ld_j^2)}\left(\Rcd(i,i) - \Rct(i,i)\right) \right. \\
&\qquad\qquad\qquad \left. + \frac{\ld_j^2(1+\ld_i^2)}{2(1+\ld_j^2)(1-\ld_i^2\ld_j^2)}\left(\Rcd(j,j) - \Rct(j,j)\right) \right] \\
&\quad + \sum_{1\leq i<j\leq n}\frac{\ld_i^2+\ld_j^2}{2(1+\ld_i^2)(1+\ld_j^2)}(\scd(i,j) + \sct(i,j)) \\
&\quad + \sum_{1\leq i<j<k\leq n} \left[ V_{ijk}(\scd(i,j) + \sct(i,j)) + V_{jki}(\scd(j,k) + \sct(j,k)) \right. \\
&\qquad\qquad\qquad\quad \left. + V_{ikj}(\scd(i,k) + \sct(i,k)) \right] \\
&\geq \tau\sum_{1\leq i<j\leq {m}}\frac{\ld_i^2+\ld_j^2}{2(1+\ld_i^2)(1+\ld_j^2)}~.
\end{align*}
for some $\tau>0$. {In the last inequality, we recall the convention that $\sct(i,k) = 0$ when $k>m$ (see the end of the second paragraph in section \ref{sec_proof_2}), and the assumption that $\scd\geq0$ when $n>m$.} By \eqref{detS1},
\begin{align*}
(\heat)\log(\det\ST) &\geq 2|A|^2 + \tau \sum_{1\leq i<j\leq {m}}\frac{\ld_i^2+\ld_j^2}{(1+\ld_i^2)(1+\ld_j^2)}
\end{align*}
provided that $f_t$ is area-decreasing.  This immediate leads to \eqref{detS3}.  With the same argument by using Lemma \ref{est_vph}, one can obtain \eqref{detS4}.

\section{Implications and Examples} \label{sec_cor}

\subsection{Dilations} \label{sec_dila}

For a map $f:(\Sm_1,g_1) \to (\Sm_2,g_2)$, pick any point $p$, and let $\ld_i$ be the singular value of $\dd f|_p$.  For any $\rho>0$, denote by $\ld_i(\rho)$ the singular value of $\dd f|_p:(T_p\Sm_1, g_1) \to (T_{f(p)}\Sm_2,\rho^2 g_2)$.  It is easy to see that
\begin{align}
\ld_i(\rho) &= \rho\,\ld_i ~.
\label{svd_dila} \end{align}
The effect of the dilation on the sectional curvature is
\begin{align}
\scc_{\rho^2g_2} &= \rho^{-2}\scc_{g_2} ~.
\label{curv_dila} \end{align}
It is worth noting that $\ld_i(\rho)\ld_j(\rho)\cdot\scc_{\rho^2g_2}$ is independent of $\rho$.

\subsection{Round Spheres}

Denote by $\gstd$ the round metric of radius $1$ on the spheres.  Theorem \ref{thm_sec_curv} can be used to derive some quantitative criterion on the homotopy class of maps between spheres.

In \cite{Gr96}*{section 2.5}, Gromov proved that there exists $\vep(n,m) > 0$ such that a map $f:(S^n,\gstd) \to (S^m,\gstd)$ with $\ld_i\ld_j<\vep(n,m)$ (for any $i\neq j$) everywhere must be homotopically trivial.  One may also consult \cite{Gu13}*{p.1882}.  In \cite{TW04}*{Corollary 1.2}, the bound is improved to be $1$.  Theorem \ref{thm_sec_curv} with the dilation trick can further improve such a bound.

\begin{cor} \label{cor_sphere}
Suppose that $n\geq m\geq 2$. Any smooth map $f:(S^n,\gstd) \to (S^m,\gstd)$ whose singular value obeys
\begin{align*}
\ld_i\ld_j < \frac{2n-m-1}{m-1}
\end{align*}
(for any $i\neq j$) everywhere must be homotopically trivial.
\end{cor}

\begin{proof}
By the compactness of $S^n$, there exists $0<\rho^{-1}<\sqrt{\frac{2n-m-1}{m-1}}$ such that $\ld_i\ld_j < \rho^{-2}$ everywhere (and for any $i\neq j$).  Consider $f:(S^n,\gstd) \to (S^m,\rho^2\gstd)$.  With the help of \eqref{svd_dila}, $f$ is area-decreasing with respect to this rescaled metric.  Due to \eqref{curv_dila}, the sectional curvature of $(S^m,\rho^2\gstd)$ is $\rho^{-2}$, which is less than $\frac{2n-m-1}{m-1}$.  Thus, Theorem \ref{thm_sec_curv} applies, and $f$ is homotopically trivial.
\end{proof}

To the best our knowledge, the bound given by Corollary \ref{cor_sphere} is sharpest so far.  What follows are some known results and remarks about the bound.
\begin{enumerate}
    \item The bound obtained in \cite{TW04} is $1$.  If one applies the results in \cites{LL11, SS14} to this setting, the bound has not been improved, and is still $1$.
    \item Recently, Assimos, Savas-Halilaj and Smoczyk \cite{ASS22} investigated the case when $m=2$.  In \cite{ASS22}*{Corollary C}, the bound is improved to $n-1$.  Note that when $m=2$, Corollary \ref{cor_sphere} gives $2n-3$.
    \item For classical Hopf fibrations, $S^{3}\to S^{2}$, $S^{7}\to S^{4}$ and $S^{15} \to S^{8}$, the singular values are $2$ and $0$, with respect to $\gstd$.  Interestingly, for these dimensions, the bounds for $\ld_i\ld_j$ given by Corollary \ref{cor_sphere} are all $3$.  The result of \cite{ASS22} only applies to $S^{3}\to S^{2}$, and the bound is $2$.  A natural question is whether the bound can be pushed to $4$ in these dimensions.
\end{enumerate}

A map from $\Sm_1\subset\BR^{n+1}$ to $\Sm_2\subset\BR^{m+1}$ is called a \emph{polynomial map} if it is the restriction of a polynomial map from $\BR^{n+1}$ to $\BR^{m+1}$.  An interesting question \cite{PT99} in algebraic topology is whether an element of $\pi_n(S^m)$ can be represented by a polynomial map between unit spheres.  The above corollary can be used to prove a bound on the degree of a polynomial map of non-trivial homotopy class.

\begin{cor}
Suppose that $n\geq m\geq 2$. A polynomial map $f:(S^n,\gstd) \to (S^m,\gstd)$ which is not homotopically trivial must have degree no less than
\begin{align*}
\sqrt{\frac{2n-m-1}{m-1}} ~.
\end{align*}
\end{cor}

\begin{proof}
According to \cite{PT02}*{Theorem 1.1}, a degree $k$ polynomial map between spheres has $\ld_i\leq k$ for all $i$ everywhere.  This corollary follows from Corollary \ref{cor_sphere}.
\end{proof}

\subsection{Complex Projective Spaces}

Consider the complex projective space endowed with the Fubini--Study metric, $\gFS$.  In terms of the homogeneous coordinate $\bfZ$,
\begin{align*}
\gFS &= \frac{|\bfZ|^2\,|\dd\bfZ|^2 - (\br{\bfZ}\cdot\dd\bfZ)(\bfZ\cdot\dd\br{\bfZ})}{|\bfZ|^4} ~.
\end{align*}
For orthonormal vectors $X$ and $Y$, the sectional curvature along $X\w Y$ is $\scc_{\gFS}(X\w Y) = 1 + 3\left(\ip{JX}{Y}\right)^2$, where $J$ is the complex structure.  The Ricci curvature is $\Ric_{\gFS} = 2(\ell+1)\,\gFS$, where $\ell$ is the complex dimension of the projective space.  Note that when $\ell = 1$, we have $\BCP^1\cong S^2$, and $\gFS = \frac{1}{4}\gstd$.

It is not hard to see that Theorem \ref{thm_Ric_curv} applies to complex projective spaces. 
\begin{cor}[Corollary \ref{cproj0}]
    Suppose that $n\geq m\geq 1$.  Any area-decreasing map from $(\BCP^n,\gFS)$ to $(\BCP^m,\gFS)$ must be homotopically trivial.
\end{cor}

As an area-decreasing map $f: \BCP^n \rightarrow \BCP^m$ with $n\geq m\geq 1$ induces a trivial map on the cohomology groups $f^*: H^2(\BCP^m) \rightarrow H^2 (\BCP^n)$, it is natural to raise the following question:

\begin{ques*}
Suppose a map $f: \BCP^n \rightarrow \BCP^m$ with $n> m\geq 1$ induces a trivial map on the cohomology groups $f^*: H^2(\BCP^m) \rightarrow H^2 (\BCP^n)$.  Is it necessarily homotopically trivial? 
\end{ques*}

We remark that in \cite{MW11}, it was proved that a pinched symplectomorphism of $(\BCP^n,\gFS)$ is homotopic to a biholomorphic isometry through the mean curvature flow. See also \cite{W13}.

The following corollary concerns $\pi_{2n+1}(\BCP^n)$, which is isomorphic to $\pi_{2n+1}(S^{2n+1})\cong\BZ$ by the long exact sequence induced by the Hopf fibration $S^1\to S^{2n+1}\to\BCP^n$.

\begin{cor}[Corollary \ref{cproj1}]
Any smooth map $f: (S^{2n+1},\gstd) \to (\BCP^n,\gFS)$ whose singular value satisfies
\begin{align*}
\ld_i\ld_j < \frac{2n}{2n+1}
\end{align*}
(for any $i\neq j$) everywhere must be homotopically trivial.
\end{cor}

\begin{proof} \label{cproj2}
Let $\rho = \sqrt{\frac{2n+1}{2n}}$.  Consider $f:(S^{2n+1},\gstd) \to (\BCP^n,\rho^2\gFS)$.  According to \eqref{svd_dila}, $f$ is area-decreasing with respect to this rescaled metric.  By \eqref{curv_dila}, $(\BCP^n,\rho^2\gFS)$ is an Einstein metric with Einstein constant $\rho^{-2}(2n+1) = 2n$.  Since the Einstein constant of $(S^{2n+1},\gstd)$ is $2n$, Theorem \ref{thm_Ric_curv} applies, and $f$ is homotopically trivial.
\end{proof}

\begin{rmk}
These discussions apply to the quaternionic projective spaces as well.  The quaternionic projective space $\BHP^n$ also carries a naturally defined Fubini--Study metric $\gFS$.  For orthonormal vectors $X$ and $Y$, the sectional curvature along $X\w Y$ is $\sec_{\gFS}(X\w Y) = 1 + 3\sum_{\mu=1}^3\left(\ip{J_\mu X}{Y}\right)^2$.  The Ricci curvature is $\Ric_{\gFS} = 4(n+2)\,\gFS$.  Note that when $\ell = 1$, we have $\BHP^1\cong S^4$, and $\gFS = \frac{1}{4}\gstd$.  It follows that the analogous statement to Corollary \ref{cproj1} holds true for $(\BHP^n,\gFS)$.  For Corollary \ref{cproj2}, the bound becomes $\frac{4n+2}{4(n+2)}$.  Again, by considering the Hopf fibration $S^3 \to S^{4n+3} \to \BHP^n$, the bound is sharp as $n\to\infty$.
\end{rmk}

\section{Comments on other approaches and earlier results}

Other than the mean curvature flow and methods that are of topological nature or rely on the $h$-principle \cites{Gr96, Gu13}, there is another approach of Llarull \cite{Llarull1998} to study area-decreasing maps by Dirac operators. 

\begin{thm}
Suppose $M$ is an $n$-dimensional compact spin Riemannian manifold with scalar curvature no less than $n(n-1)$. Any area-non-increasing map $f:M \to S^n$  with non-zero degree is an isometry. 
\end{thm}

In view of Corollary  \ref{cproj0}, it is natural to speculate that the mean curvature flow will give a classification of the homotopy class of area-non-increasing maps from $(\BCP^n,\gFS)$ to $(\BCP^n,\gFS)$.

In the rest of this section, we explain how Theorem \ref{thm_detS} also implies \cite{SS14}*{Theorem A}.  In particular, we show that the term $\CR_S$ in Theorem \ref{thm_detS} is non-negative under the assumptions of  \cite{SS14}*{Theorem A}: $n\geq 2$, and there exists a constant $\sm>0$ such that
\begin{align}
\scd &> -\sm \qquad\text{and}\qquad \Rcd \geq (n-1)\sm \geq (n-1)\sct ~.
\label{asm_SS} \end{align}

As before, start with \eqref{detS1}, and rewrite \eqref{RS0}.  Since $\sct \leq \sm$,
\begin{align*}
\CR_S &\geq \sum_{1\leq i<j\leq n} (S_{ii}+S_{jj})^{-1} \left[ \frac{2\ld_i^2}{(1+\ld_i^2)^2} \sum_{\substack{1\leq k\leq n\\k\neq i}}\left( \frac{1}{1+\ld_k^2}\,\scd(i,k) - \frac{\ld_k^2}{1+\ld_k^2}\sm \right) \right. \\
&\qquad\qquad\qquad\qquad\qquad \left. + \frac{2\ld_j^2}{(1+\ld_j^2)^2} \sum_{\substack{1\leq k\leq n\\k\neq j}}\left( \frac{1}{1+\ld_k^2}\,\scd(j,k) - \frac{\ld_k^2}{1+\ld_k^2}\sm \right) \right] ~.
\end{align*}
With $\frac{2}{1+\ld_k^2} = 1 + S_{kk}$ and $\frac{2\ld_k}{1+\ld_k^2} = 1 - S_{kk}$, $\CR_S$ is no less than
\begin{align*}
&\quad \sum_{1\leq i<j\leq n} (S_{ii}+S_{jj})^{-1} \left[ \frac{\ld_i^2}{(1+\ld_i^2)^2} \sum_{\substack{1\leq k\leq n\\k\neq i}} \big( (\scd(i,k)-\sm) + S_{kk}(\scd(i,k)+\sm) \big) \right. \\
&\qquad\qquad\qquad\qquad\qquad \left. + \frac{\ld_j^2}{(1+\ld_j^2)^2} \sum_{\substack{1\leq k\leq n\\k\neq j}} \big( (\scd(j,k)-\sm) + S_{kk}(\scd(j,k)+\sm) \big) \right] \\
&= \sum_{1\leq i<j\leq n}(S_{ii}+S_{jj})^{-1} \left[ \frac{\ld_i^2}{(1+\ld_i^2)^2}\left(\Rcd(i,i) - (n-1)\sm\right) + \frac{\ld_j^2}{(1+\ld_j^2)^2}\left(\Rcd(j,j) - (n-1)\sm\right) \right] \\
&\quad + \sum_{1\leq i<j\leq n}(S_{ii} + S_{jj})^{-1} \left[ \frac{\ld_i^2}{(1+\ld_i^2)^2}\sum_{\substack{1\leq k\leq n\\k\neq i}}S_{kk}(\scd(i,k) + \sm)  + \frac{\ld_j^2}{(1+\ld_j^2)^2}\sum_{\substack{1\leq k\leq n\\k\neq j}}S_{kk}(\scd(j,k) + \sm) \right] ~.
\end{align*}
For the latest summation, one can apply the exact same argument as that for \eqref{Ricb}.  It follows that
\begin{align*}
\CR_S &\geq \sum_{1\leq i<j\leq n} \left[ \frac{\ld_i^2(1+\ld_j^2)}{2(1+\ld_i^2)(1-\ld_i^2\ld_j^2)}\left(\Rcd(i,i) - (n-1)\sm\right) \right. \\
&\qquad\qquad\qquad \left. + \frac{\ld_j^2(1+\ld_i^2)}{2(1+\ld_j^2)(1-\ld_i^2\ld_j^2)}\left(\Rcd(j,j) - (n-1)\sm\right) \right] \\
&\quad + \sum_{1\leq i<j\leq n}\frac{\ld_i^2+\ld_j^2}{2(1+\ld_i^2)(1+\ld_j^2)}(\scd(i,j) + \sm) \\
&\quad + \sum_{1\leq i<j<k\leq n} \left[ V_{ijk}(\scd(i,j) + \sm) + V_{jki}(\scd(j,k) + \sm) + V_{ikj}(\scd(i,k) + \sm) \right] ~,
\end{align*}
where $V_{ijk}$ is given by \eqref{fct_V}.  Under assumption \eqref{asm_SS}, the rest of the argument is the same as that for the proof of Theorem \ref{thm_Ric_curv}.


\end{document}